\documentclass[a4paper,12pt]{article}
\usepackage[centertags]{amsmath}
\usepackage{amsfonts}
\usepackage{amssymb}
\usepackage{amsthm}
\usepackage{dsfont}
\usepackage{tikz, subfigure}
\usepackage{verbatim}

\newtheorem{theorem}{Theorem}[section]
\newtheorem{lemma}[theorem]{Lemma}
\newtheorem{cor}[theorem]{Corollary}
\newtheorem{prop}[theorem]{Proposition}

\renewcommand{\a}{{\underline{a}}}
\renewcommand{\b}{{\underline{b}}}
\newcommand{\p}{{\underline{p}}}
\newcommand{\q}{{\underline{q}}}

\begin{document}
\title{The dimension of projections of self-affine sets and measures}
\author{Kenneth Falconer and Tom Kempton}
\maketitle
\setlength{\parskip}{0.3cm} \setlength{\parindent}{0cm}

\begin{abstract} 
Let $E$ be a plane self-affine set defined by affine transformations with linear parts given by matrices with positive entries. We show that if $\mu$ is a Bernoulli measure on $E$ with $\dim_H\mu = \dim_L\mu$, where $\dim_H$ and $\dim_L$ denote  Hausdorff and Lyapunov dimensions,  then the 
projection of $\mu$ in all but at most one direction has Hausdorff dimension $\min\{\dim_H \mu, 1\}$. We transfer this result to sets and show that many self-affine sets have projections of dimension $\min\{\dim_H E, 1\}$ in all but at most one direction\footnote{This research was supported by EPSRC grant EP/K029061/1}.

\end{abstract}

\section{Introduction}
\setcounter{equation}{0}
Marstrand's projection theorem states that, given a set $E\subset\mathbb R^n$, for almost every $m$-dimensional linear subspace $K$ of $\mathbb R^m$ 
\begin{equation}\label{projdim}
\dim_H \pi_K E =  \min\{\dim_H E, m\},
\end{equation}
where  $\pi_K$ denotes orthogonal projection onto $K$ and $\dim_H$ denotes Hausdorff dimension, see \cite{FalBk, Marstrand, Mat}. The analogous projection theorem for measures, generally proved using potential theoretic methods, see \cite{Kau}, is that for a finite Borel measure $\mu$ on $\mathbb{R}^n$,
\begin{equation}\label{projdimmes}
\dim_H \pi_K \mu =  \min\{\dim_H \mu, m\},
\end{equation}
for almost all subspaces $K$, where the  {\it (lower) Hausdorff dimension} of a measure is given in terms of dimensions of sets by
$$\dim_H\nu=\inf\{\dim_H U:\nu(U)=1\}.$$
However, these projection theorems are no help in identifying the subspaces $K$, if any, for which \eqref{projdim} or  \eqref{projdimmes} fail. Recently there has been a great deal of interest in finding the  `exceptional directions' for projections of fractals and fractal measures of various types, especially those with a recursive structure or which are dynamically generated, see the survey \cite{FalFrX} and references therein. In particular, various classes of measures and sets have been shown to have every projection satisfying \eqref{projdim}.
For example, for self-similar sets satisfying the open set condition, \eqref{projdim} holds for all subspaces $K$ provided that the group generated by the orthonormal components of the defining similarities is dense in the orthogonal group $O(n)$, with similar results for measures, see \cite{FalXio,PerSher, HochShmerProj}. On the other hand, if the group is not dense then there are always some exceptional directions, see \cite{Far}.

It is natural to ask under what circumstances  \eqref{projdim} holds for self-affine sets and \eqref{projdimmes} for self-affine measures (by which we mean Bernoulli measures on self-affine sets) in all or virtually all directions. Dimensional analysis in the self-affine case is more awkward than in the self-similar case, not least because the dimensions of the sets or measures themselves are not continuous in the defining affine transformations. Nevertheless, the dimensions of many specific self-affine sets and also of generic self-affine sets in certain families are now known. In particular, the Hausdorff dimension $\dim_H E$ of a self-affine set $E$  `often' equals its affinity dimension $\dim_A E$, defined in \eqref{dimaff1} - \eqref{dimaff} below, see \cite{FalconerAffineSurvey}. 

Much of the work to date on projections of self-affine sets concerns projections of plane self-affine carpets, that is where the defining affine maps preserve both horizontal and vertical axes; see \cite{Almarza,FFS} for various examples. Perhaps not surprisingly, for many carpets  \eqref{projdim} fails only for projections in the horizontal and/or  vertical directions.

In this paper we consider projections of plane self-affine sets measures where the linear parts of the self-affine maps may be represented by matrices with all entries positive, in other words where the linear parts map the positive quadrant strictly into itself. 
We show that in many cases projections  onto lines in {\em all} directions (in some cases all but one direction) satisfy  \eqref{projdim} or  \eqref{projdimmes}. Many self-affine sets and measures fit into this context, both in the `generic setting' (provided that the affine maps are sufficiently contracting), see \cite{FalconerAffinity1,Sol}, and for many specific cases including those presented in \cite{Barany, FalKemp,HueLal}. 

To study projections of a set one normally examines the projections of a suitable measure on the set. Thus, we first study projections of a Bernoulli measure $\mu$ supported by a self-affine set $E$. Then  $\mu$  induces a measure $\mu_F$, known as a Furstenberg measure, on the space of line directions. Recent results on the  dynamical structure  of self-affine sets \cite{Barany, FalKemp} imply that the projections of $\mu$ are exact dimensional in  $\mu_F$-almost all directions. Together with  lower bounds for the dimension of projections from \cite{HochShmerProj} involving $r$-scale entropies we conclude that the dimensions of projections of measures are well-behaved in all but perhaps one direction. The results on projections of self-affine sets follow by supporting suitable Bernoulli measures on $E$.

\section{Preliminaries}
\setcounter{equation}{0}
In this section we introduce the self-affine sets that we consider along with pertinent notation. 
Let $A_1,\cdots, A_k$ be a collection of  $2\times 2$ matrices of Euclidean norm less than $1$ with strictly positive entries, and let $d_1,\cdots, d_k\in\mathbb R^2$. Then the maps  $T_i:\mathbb R^2\to \mathbb R^2 \ (i =1,\cdots, k)$  given by
\begin{equation}\label{ifsmaps}
T_i(x)=A_i(x)+d_i \qquad (x\in\mathbb R^2) 
\end{equation}
are affine contractions which form an iterated function system (IFS).  By standard  IFS theory, see \cite{FalBk, Hut} there exists  a unique non-empty compact set $E$ satisfying
\[
E=\bigcup_{i=1}^k T_i(E),
\]
termed a  {\it self-affine set}. 

There is a natural coding on $E$ and its usual iterated construction.
Let $\Lambda:=\{1,\cdots,k\}$, let 
$\Sigma^*= \bigcup_{n=0}^\infty \Lambda^n$ be the space of finite words formed by elements of $\Lambda$, and let $\Sigma= \Lambda^{\mathbb N}$ be the corresponding space of infinite words. 
Associated with each  $a_1\cdots a_n \in \Sigma^*$ is the {\it cylinder}
$$[a_1\cdots a_n] = \{a_1 \cdots a_n b_{n+1} b_{n+2}, \cdots : b_i \in \Lambda\}\subset \Sigma;$$
 the cylinders form a basis of open and closed sets for the natural topology on $\Sigma$. 
 
We abbreviate compositions of the $T_i$ by $T_{a_1\cdots a_n}:=T_{a_1}\circ\cdots \circ T_{a_n}$ for $a_1 \cdots a_n \in \Sigma^*$.
Let $D\subset \mathbb R^2$ be the unit diskso that 
$$D_{a_1\cdots a_n} := T_{a_1\cdots a_n}(D).$$
are  the ellipses obtained as  images of $D$ under compositions of the $T_i$.
  By scaling we may assume that $T_i(D) \subset D$ for all $i$, in which case the self-affine set $E$ may be written as
 $$ E = \bigcap_{n=0}^\infty \bigcup_{a_1\cdots a_n}  T_{a_1\cdots a_n}(D).$$

Let $\mathbb P\mathbb R^1 = (\mathbb R^2\setminus 0)/\sim$ be 1-dimensional real projective space where $\sim$ identifies points on each line through the origin. We may parameterise $\mathbb P\mathbb R^1$ by the angle between lines  and the horizontal axis, and these angles define a metric $d$ on $\mathbb P\mathbb R^1$ in the obvious way.  Each matrix $A_i^{-1}$ induces a projective linear transformation $\phi_i: \mathbb P\mathbb R^1\to\mathbb P\mathbb R^1$ given by $\phi_i:= A_i^{-1}/\sim$, that is, $A_i^{-1}$ maps straight lines at angle $\theta$ to straight lines at angle $\phi_i(\theta)$. Since the matrices $A_i$ are strictly positive, each $\phi_i$ restricted to the negative quadrant ${\mathcal Q}_2$ is a contraction so that the system of contractions $(\phi_i)_{i=1}^k$ is an iterated function system on ${\mathcal Q}_2$. Let $F$ be the attractor of this IFS, that is the non-empty compact subset of ${\mathcal Q}_2$ such that $F=\bigcup_{i=1}^k T_i(F)$,

Now let $\mu$ be a Bernoulli probability measure on $\Sigma$; for notational convenience    
we also write $\mu$ for the corresponding measure on $E$, that is its image under the map
$(a_1 a_2 \cdots) \mapsto  \bigcap_{n=0}^\infty  T_{a_1\cdots a_n}(D)$.

The {\it Furstenberg measure} $\mu_F$ is defined to be the stationary probability measure supported by  $K \subset \mathbb{RP}^1$ associated to the maps $\phi_i$ chosen according to the measure $\mu$, see for example,  \cite{BPS}. Specifically, $$
\mu_F(U)=\sum_{i=1}^k \mu[i]\mu_F(\phi_i^{-1}(U)),
$$
for all Borel sets $U$.
The  Furstenberg measure is central to the proofs of our results.

Ideally we would like to be able to express the Hausdorff dimension of a self-affine set $E$ in terms of its defining IFS maps \eqref{ifsmaps}, as can be done for self-similar sets subject to a separation condition. (For the definition and basic properties of Hausdorff dimension, which we denote by  $\dim_H$,  see for example \cite{FalBk, Mat}.) A natural candidate for $\dim_H E$ is the affinity dimension defined in \cite{FalconerAffinity1} in terms of the linear parts of the IFS maps.
Let $\alpha_1(A)\geq \alpha_2(A)\geq 0$ be the {\it singular values} of a linear mapping or matrix  $A$ on $\mathbb R^2$, that is  the lengths of the major and minor semiaxes of the  ellipse
$A(D)$ or equivalently the positive square roots of the eigenvalues of $AA^T$. For $0\leq s\leq 2$ we define the {\it singular value function} of $A$ by 
$$\phi^s(A)=\left\lbrace
\begin{array}{cl}
\alpha_1^s & 0<s\leq 1\\
\alpha_1\alpha_2^{s-1} & 1\leq s\leq 2\\
(\det A)^{s/2} & 2\leq s
 \end{array}\right. .$$
The submultiplicativity of the singular value functions enables us to define the {\it affinity dimension} of a set $E\subset \mathbb R^2$ defined by the IFS \eqref{ifsmaps}  by
\begin{equation}\label{dimaff1}
\dim_{A}(A_1,\cdots,A_k)\equiv    \dim_{A}E=\Big\{s:\lim_{n\to\infty}\big(\sum_{a_1\cdots a_n\in\Lambda^n}\phi^s(A_{a_1\cdots a_{n}})\big)^{1/n} = 1\Big\}
\end{equation}
or equivalently
\begin{equation}\label{dimaff}
\dim_{A}(A_1,\cdots,A_k) \equiv \dim_{A}E=\inf\Big\{s:\sum_{n=1}^{\infty}\sum_{a_1\cdots a_n\in\Lambda^n}\phi^s(A_{a_1\cdots a_{n}})<\infty\Big\}.
\end{equation}
Note that the affinity dimension depends on the $\{A_i\}$ that occur in the iterated function theorem that defines $E$, rather than on $E$ itself, but referring to $\dim_AE$ generally causes little problem. 
For every self-affine set $E$ it is the case that $\dim_{H}E \leq \dim_{A}E$, but equality holds in many cases, both generic and specific, see the survey \cite{FalconerAffineSurvey} and recent examples in \cite{Barany, FalKemp}. 

We will also refer to  Lyapunov dimension, which reflects the geometry of the self-affine set $E$ relative to the measure $\mu$. The {\it Lyapunov exponents} $\lambda_1(\mu), \lambda_2(\mu)$ are constants such that, for $\mu$-almost every $a_1 a_2 \cdots \in \Sigma$,
\begin{equation}\label{lyap}
\lim_{n\to\infty}\frac{1}{n}\log \alpha_i(a_1 \cdots a_n)=\lambda_i \qquad (i = 1,2).
\end{equation}
The {\it Lyapunov dimension} of $\mu$ is given by
\begin{equation}\label{lyapdim}
\dim_L \mu:= \left\lbrace
\begin{array}{cc}
\dfrac{h_{\sigma}(\mu)}{-\lambda_1(\mu)} & \quad h_{\sigma}(\mu)\leq -\lambda_1(\mu)\\
1+\dfrac{h_{\sigma}(\mu) + \lambda_1(\mu)}{-\lambda_2(\mu)}& \quad h_{\sigma}(\mu)\geq -\lambda_1(\mu)
\end{array}\right. ,
\end{equation}
where $h_{\sigma}(\mu)$ is the Kolmogorov-Sinai entropy of the system $(\Sigma,\sigma,\mu)$ with $\sigma$  the left shift on $\Sigma$.
Note that $\dim_L \mu$ depends only on the matrices $\{A_1,\cdots, A_k\}$ and the measure $\mu$. It is always the case that $\dim_H(\mu)\leq \dim_L(\mu)$, see \cite{BPS}.

To obtain our dimension estimates, we will use  $r$-{\it scale entropies}, given by
$$H_r(\nu):= -\int_{[-1,1]}\log \nu(B_r(x)) d\nu(x)$$
where $\nu$ is a probability measure on $[-1,1]$ and  $r>0$.
A  measure $\nu$ is termed {\it exact dimensional of dimension} $\beta$ if
\begin{equation}\label{exact}
\lim_{r \to 0} \frac{\log \nu(B_r(x))}{\log r} = \beta
\end{equation}
for $\nu$-almost all $x$. In particular, if $\nu$ is exact dimensional then 
\begin{equation}\label{exactent}
\dim_H \mu = \beta = \lim_{r\to 0} \frac{H_r(\nu)}{-\log r}.
\end{equation}

\section{Statement of Results}
\setcounter{equation}{0}
We define $B\subset\mathbb P\mathbb R^1$ by $\mathbb{PR}^1\setminus\{\theta\}$ if all matrices $A_i$ have a common maximal eigenvector in some direction $\theta$, and by $B=\mathbb {PR}^1$ otherwise. The set $\mathbb{PR}^1\setminus B$ consists of at most one angle and represents the set of possible exceptional directions for our projection theorems.

It was proved in \cite[Proposition 3.3]{Barany} that there exists a constant $\beta(\mu)$ such that for $\mu_F$-almost every $\theta$ the projected measure $\pi_{\theta}(\mu)$ is exact dimensional with dimension $\beta(\mu)$. This allows us to state our main theorem.

\begin{theorem}\label{MainThm}
Let $\mu$ be a Bernoulli measure on a planar self-affine set defined by an IFS \eqref{ifsmaps} with strictly positive matrices $A_i$. Then for all $\theta\in B$
\[
\dim_H  \pi_{\theta}(\mu)\geq\beta(\mu).
\]
\end{theorem}
This yeilds the following corollaries
\begin{cor}
Let $\mu$ be a Bernoulli measure on a planar self-affine set associated to strictly positive matrices $A_i$. Suppose $\dim_H \mu=\dim_L \mu $. Then for all $\theta\in B$,
\[
\dim_H \pi_{\theta}(\mu)=\min\{\dim_H \mu,1\}.
\] 
\end{cor}
\begin{proof}
If we can show that
\begin{equation}\label{dimdrop}
\dim_H\mu=\dim_L\mu\implies \beta(\mu)=\min\{\dim_H \mu,1\}.
\end{equation}
the conclusion follows directly from Theorem \ref{MainThm}.
Under the additional assumption that the self-affine set satisfies the strong separation condition,  \eqref{dimdrop} along with the converse implication was established in \cite{FalKemp}.
However, the strong separation condition is not required for the implication \eqref{dimdrop}. This is because the left hand side of inequality (4.2)  of \cite[Lemma 4.2]{FalKemp} holds even without strong separation. 
\end{proof}

Theorem \ref{MainThm} can be applied to the projection of self-affine sets. The next corollary shows that if the set supports Bernoulli measures with dimensions approximating the affinity dimension   there is no dimension drop for projections of $E$ except possibly in one direction.

\begin{cor}\label{corproj}
Let $E$ be a planar self-affine set defined by an IFS \eqref{ifsmaps} with strictly positive matrices $A_i$. If there exists a sequence $(\mu_n)$ of Bernoulli measures supported on $E$ with $\dim_H \mu_n \to \dim_A E$ then, for all $\theta\in B$, 
\begin{equation}\label{cor3.3}
\dim_H\pi_{\theta}(E)=\min\{\dim_H E,1\}.
\end{equation}
\end{cor}

\begin{proof}
Let $\mu_{n,F}$ denote the Furstenberg measure associated to $\mu_n$, and let $\beta_n$ denote the value of  $\dim_H \pi_{\theta}(\mu_n)$  that occurs for $\mu_{n,F}$-almost every $\theta$ by \cite[Proposition 3.3]{Barany}. 
Then Theorem \ref{MainThm} implies that for all $\theta\in B$,
\[
\dim_H \pi_{\theta}(E)\geq 
\lim_{n\to\infty}\beta_n = \min\{\dim_A E,1\},
\] 
this last equality holding because $\dim_H \mu_n\to \dim_A E$ and using \cite[Theorem 2.7]{Barany}.
Since for all $\theta$
$$\min\{\dim_A E,1\}\geq \min\{\dim_H E,1\} \geq \dim_H \pi_{\theta}(E),$$ 
\eqref{cor3.3} follows.
\end{proof}

It would be nice to replace the condition in Corollary \ref{corproj} on the existence of the sequence of Bernoulli measures by the requirement that $\dim_A E=\dim_H E$. However the relationship between the statements `$\dim_A E=\dim_H E$', `$E$ supports an ergodic measure $\mu$ with $\dim_H \mu=\dim_A E$' and `there exists a sequence $(\mu_n)$ of Bernoulli measures supported on $E$ with $\dim \mu_n \to \dim_A E$' is not clear. All three statements hold for almost all sets of translation vectors $(d_1\cdots d_k)$ whenever each $A_i$ has $||A_i||<\frac{1}{2}$, but the relationship between the statements is difficult to understand.

Note that our results require the matrices $A_i$ to have positive entries, so that the matrices all map the first quadrant into its interior. This assumption is used in various ways and shortens some of the proofs, but it is crucial in B\'ar\'any's proof \cite[Proposition 3.3]{Barany} that projections of self-affine measures are exact dimensional, which in turn is used in the proof of Theorem \ref{Thm2} on averages of $r$-scale entropies.

\subsection{Examples}
There are many families of measures on self-affine sets for which $\dim_H\mu=\dim_L\mu$ and hence for which our main theorem holds. In particular, this includes  classes of examples presented in Hueter and Lalley \cite{HueLal}, B\'ar\'any \cite{Barany} and Falconer and Kempton \cite{FalKemp}. The  corollary below shows that this situation also arises for almost all sets of translations in the affine maps in the IFS that defines $E$.

\begin{cor}
Let $\{T_i\}_{i=1}^k$ be an affine  IFS \eqref{ifsmaps} where the matrices  $A_i$ are  strictly positive with $\|A_i\| < \frac{1}{2}$ for all $i$. For each  translation vector ${\bf d}= (d_1,\cdots,d_k)\in \mathbb{R}^{2k}$ let $E_{\bf d}$ be the self-affine attractor thus defined. Then for almost all ${\bf d}\in \mathbb{R}^{2k}$ $($in the sense of $2k$-dimensional Lebesgue measure$)$
\[
\dim_H\pi_{\theta}(E_{\bf d})=\min\{\dim_H E_{\bf d},1\} = \min\{ \dim_A E_{\bf d},1\}
\] 
for all $\theta \in B$ simultaneously $($where as before $B$ is either $\mathbb{PR}^1$ or is $\mathbb{PR}^1$ with one angle omitted$)$.

\end{cor}

\begin{proof}
The right-hand equality follows since  $\dim_HE_{\bf d} = \min\{ \dim_A E_{\bf d},2\}$ for almost all ${\bf d}$, see \cite{FalconerAffinity1,Sol}.

For the left-hand equality we will show that  for all $\epsilon>0$ we can find a Bernoulli measure $\mu$ on $\Sigma$ that induces measures  $\mu_{\bf d}$ on the $E_{\bf d}$ with $\dim_H\mu_{\bf d}>\dim_H E_{\bf d}- \epsilon$ for almost all ${\bf d}$, so that the conclusion will follow from Corollary \ref{corproj}.

As the $A_i$ have strictly positive entries, there is a cone $K$ strictly inside the upper-right quadrant that is mapped into itself by all finite compositions of the $A_i$. This implies, see \cite[Lemma 3.1]{Shmer}, that there is a number $c\geq 1$ that depends only on $K$ such that for all finite words 
$\a=a_1\cdots a_j,\ \b=b_1\cdots b_{j'} \in \Sigma^*$,
$$ \|A_\a A_\b\| \leq  \|A_\a\| \| A_\b\| \leq c\,  \|A_\a A_\b\|.$$
Since $\phi^s(A) = \|A\|^{s} \ (0\leq s\leq 1)$ and  
$\phi^s(A) = \|A\|^{2-s} (\det A)^{s-1}\ (1\leq s \leq 2)$, it follows that 
\begin{equation}\label{phiin}
\phi^s(A_\a A_\b) \leq  \phi^s(A_\a) \phi^s( A_\b) \leq c \,  \phi^s(A_\a A_\b).
\end{equation}
for all $s\geq 0$.

Write $d_A:=  \dim_A (A_1,\cdots,A_k) \equiv  \dim_A E_{\bf d}$ for the affinity dimension. We may assume that $0< d_A \leq 2$ (if $d_A> 2$ then $E_{\bf d}$ has positive plane Lebesgue measure for almost all ${\bf d}$ and the result is clear). Let $0< t<d_A$ ($t$ not an integer). Using a simple estimate of the rate of decrease of  $\phi^{s}(A_\a)$ with $s$, see  \cite{FalconerAffinity1}, we may choose  an $N \in \mathbb{N}$ sufficiently large  to ensure that for some $0<\lambda<1$
\begin{equation}\label{phiratios}
c\ \frac{\phi^{d_A}(A_\a)}{\phi^{t}(A_\a)} < \lambda \qquad \mbox{for all } \a = a_1\cdots a_N \in \Lambda^N.
\end{equation}

From the definition of $d_A$ \eqref{dimaff1} and submultiplicative properties,   
$\inf_{n\in \mathbb{N}} \big(\sum_{|\a| = n} {\phi^{d_A}(A_\a)}\big)^{1/n} = 1$, so we may choose
$s\geq d_A$ such that 
\begin{equation}\label{sumone}
\sum_{\a \in \Lambda^N} {\phi^{s}(A_\a)}=\big(\sum_{\a \in \Lambda^N} {\phi^{s}(A_\a)}\big)^{1/N}= 1
\end{equation}
 using the continuity of these sums in $s$.

Let $\Sigma^*_N = \bigcup_{j=0}^\infty \Lambda^{jN}$.   We define a Bernoulli measure $\mu$ on $\Sigma$ by specifying $\mu$ on cylinders defined by words in  $\Sigma^*_N$:
$$\mu([\a_1\cdots \a_n]) = \phi^{s}(A_{\a_1})\phi^{s}(A_{\a_2})\cdots \phi^{s}(A_{\a_n})
\qquad (n \in \mathbb{N},\ \a_i \in \Lambda^{N}).$$
By \eqref{sumone} $\mu$ defines a Borel probability measure on $\Sigma= (\Lambda^{N})^\mathbb{N}$. Using \eqref{phiin} and \eqref{phiratios}
\begin{equation}
\frac{\mu([\a_1\cdots \a_n])}{\phi^t(A_{\a_1\cdots \a_n})}
\leq \frac{\phi^s(A_{\a_1})\cdots\phi^s(A_{\a_n})}{c^{1-n}\phi^t(A_{\a_1})\cdots\phi^t(A_{\a_n})}
\leq \frac{c^{n-1}\phi^{d_A}(A_{\a_1})\cdots\phi^{d_A}(A_{\a_n})}{\phi^t(A_{\a_1})\cdots\phi^t(A_{\a_n})} \leq c^{-1}\lambda^n. \label{comp}
\end{equation}


We now proceed as in the proof of \cite[Theorem 5.3]{FalconerAffinity1}. We represent points of $E_{\bf d} $ in the usual way by  $x_{\bf d}(\a) = \cap_{i=0}^\infty  T_{\a|_i}(D) $ for  $\a \in \Sigma$. We write $\mu_{\bf d}$ for the push-down of $\mu$  onto $E_{\bf d} $, given by $\mu_{\bf d}(F) = \mu\{\a: x_{\bf d}(\a) \in F\}$. 
We write
$\a\wedge\b$ for the common initial word of $\a,\b \in \Sigma$.   
 As in \cite[Theorem 5.3]{FalconerAffinity1}  we may bound the energy integral of $\mu_{\bf d}$ integrated over some disc $B \subset \mathbb{R}^{2k}$ by
\begin{eqnarray*}
\int_{{\bf d} \in B}\int_{x \in E_{\bf d}}\int_{y \in E_{\bf d}}
\frac{d\mu_{\bf d}(x) d\mu_{\bf d}(y)d{\bf d}}{|x-y|^t}
&= & 
\int_{{\bf d} \in B}\int_{\a \in \Sigma}\int_{\b \in \Sigma}
\frac{d\mu(\a) d\mu(\b) d{\bf d}}{|x_{\bf d}(\a)-x_{\bf d}(\b)|^t}\\
&\leq & 
c_1\int_{\a \in \Sigma}\int_{\b \in \Sigma}\phi^t(A_{\a\wedge \b})^{-1} d\mu(\a) d\mu(\b)\\
&\leq & 
c_1 \sum_{\p \in \Sigma^*}   \phi^t (A_{\p})^{-1}\mu([\p])^2\\
&\leq & 
c_1c_2 c_3\sum_{\q \in\Sigma^*_N}   \phi^t (A_{\q})^{-1} \mu([\q])^2.
\end{eqnarray*}
For the last inequality we have regarded each word $\p \in \Sigma^*$ as $\p = \q q_1\cdots q_j$ with $\q \in\Sigma^*_N$ and $0\leq j \leq N-1$, and used that  $\phi^t (A_{\p})^{-1}\mu([\p])^2 \leq c_2\phi^t (A_{\q})^{-1}\mu([\q])^2$ for some $c_2$ independent of $\p$, since the summands increase by a bounded factor on adding each single letter to a word $\q$. The constant 
$c_3 = \sum_{i=0}^{N-1} k^i$ is the number of words $\p \in \Sigma^*$ for which the summands are estimated by the summand of each $\q \in\Sigma^*_N$. Writing $c_4 = c_1c_2 c_3$ and using \eqref{comp}, 
\begin{eqnarray*}
\int_{{\bf d} \in B}\int_{x \in E_{\bf d}}\int_{y \in E_{\bf d}}
\frac{d\mu_{\bf d}(x) d\mu_{\bf d}(y)d{\bf d}}{|x-y|^t}
&\leq & 
c_4\sum_{j=0}^\infty\sum_{\q \in \Lambda^{jN}}   \phi^t (A_{\q})^{-1} \mu([\q]) \mu([\q])\\
&\leq & 
c_4c^{-1}\sum_{j=0}^\infty \lambda^j\sum_{\q \in \Lambda^{jN}}  \mu([\q]) \\
&= & 
c_4c^{-1}\sum_{j=0}^\infty \lambda^j \ < \ \infty.
\end{eqnarray*}
We conclude that for almost all ${\bf d}$,\  $  \int_{x \in E_{\bf d}}\int_{y \in E_{\bf d}}
|x-y|^{-t} d\mu_{\bf d}(x) d\mu_{\bf d}(y) <\infty$ implying that for $\mu_{\bf d}$-almost all $x \in E_{\bf d}$ there is a constant $c_0$ such that   $\mu_{\bf d} B(x,r) \leq c_0 r^t$ for all $r>0$, so $\dim\mu_{\bf d} \geq t$. This is the case for all $t<d_A$, as required. 
\end{proof}

\section{Proof of the Main Theorem}
\setcounter{equation}{0}
In this section we prove Theorem \ref{MainThm}. There are two intermediate results which have relatively technical proofs whose main ideas are already well-understood in other work. Since these proofs are quite long and somewhat tangential to the thrust of our argument, we defer them to the appendix.

We first recall a proposition from \cite{FalKemp} on the dynamics of projections of self-affine sets. Let $\pi_{\theta}:D\to[-1,1]$ denote orthogonal projection in direction $\theta$ onto the diameter of $D$ at angle perpendicular to $\theta$, where we identify this diameter (of length 2) with the line $[-1,1]$. 

\begin{prop}\label{SelfSimilarFamily}
For each $i\in\{1,\cdots,k\}, \theta\in\mathbb P\mathbb R^1$ there is a well defined linear contraction $f_{i,\theta}:[-1,1]\to[-1,1]$ given by
\[
f_{i,\theta}=\pi_{\theta}\circ T_i\circ \pi_{\phi_i(\theta)}^{-1}
\]
such that $\pi_{\theta}(T_i(A))=f_{i,\theta}(\pi_{\phi_i(\theta)}(A))$ for all Borel sets $A\subset D$.
\end{prop} 
This identity represents $\pi_{\theta}(E)$ as a union of scaled-down copies of projections of $E$ in other directions.

We let $\Sigma^*$ denote the space of all finite words made by concatenating the letters $\{1,\cdots,k\}$. We think of $\Sigma^*$ as a tree, letting words $a_1\cdots a_{n+1}$ be children of the parent node $a_1\cdots a_n$. Given a direction $\theta$ in which we want to project $E$, we define a length function $|.|_{\theta}$ on the tree $\Sigma^*$ by declaring the length of the path from the root to the node $a_1\cdots a_n$ to be
\[
|a_1\cdots a_n|_{\theta}=-\log_2\left(\frac{|\pi_{\theta}(D_{a_1\cdots a_n})|}{2}\right).
\]
Since $|\pi_{\theta}(D)|=2$, the division by $2$ inside the log is necessary to ensure that the length of the empty word is zero.

\begin{lemma}
The length function $|.|_{\theta}$ satisfies
\[
|a_1\cdots a_{n+1}|_{\theta}-|a_1\cdots a_n|_{\theta}=|a_{n+1}|_{\phi_{a_n\cdots a_1}(\theta)}.
\]
\end{lemma}
\begin{proof}
Identically, 
\[
\log_2\left(\frac{|\pi_{\theta}(D_{a_1\cdots a_{n+1}})|}{2}\right)=\log_2\left(\frac{|\pi_{\theta}(D_{a_1\cdots a_n})|}{2}\right)+\log_2\left(\frac{|\pi_{\theta}(D_{a_1\cdots a_{n+1}})|}{|\pi_{\theta}(D_{a_1\cdots a_n})|}\right).
\]
Applying the linear map $T_{a_1\cdots a_n}^{-1}$ to $D_{a_1\cdots a_n}$ and using Proposition \ref{SelfSimilarFamily},
\begin{align*}
\log_2\left(\frac{|\pi_{\theta}(D_{a_1\cdots a_{n+1}})|}{|\pi_{\theta}(D_{a_1\cdots a_n})|}\right)&=\log_2\left(\frac{|\pi_{\phi_{a_n\cdots a_1}}(D_{a_{n+1}})|/2}{|\pi_{\phi_{a_n\cdots a_1}}(D)|/2}\right)\\
&= \log_2\left(\frac{|\pi_{\phi_{a_n\cdots a_1}}(D_{a_{n+1}})|}{2}\right)
\end{align*}
since $|\pi_{\theta}(D)|/2=1$ for all $\theta$.
\end{proof}
 
For each $\underline a\in\Sigma$ and $N\in\mathbb N$ we define $n_j=n_j(\underline a, N,\theta)$ to be the natural number satisfying
\begin{equation}\label{nk}
|a_1\cdots a_{n_j-1}|_{\theta}< Nj\leq|a_1\cdots a_{n_j}|_{\theta}.
\end{equation}

We define the infinite tree $\Sigma_N=\Sigma_N(\theta)\subset (\Sigma^*)^{\mathbb N}$ to be the tree with nodes at level $j$ labeled by words $a_1\cdots a_{n_j}$, and edges between level $j-1$ nodes and level $j$ nodes labelled by $a_{n_{j-1}+1}\cdots a_{n_j}$. 

A word $\underline a\in\Sigma$ can also be regarded as an element of $\Sigma_N$ in which the first symbol is $a_1\cdots a_{n_1}$, the second is $a_{n_1+1}\cdots a_{n_2}$ etc. When the distinction between $\Sigma$ and $\Sigma_N$ is important we specify which space $\underline a$ is in. The measure $\mu$ on $\Sigma$ may also be regarded as a measure on $\Sigma_N$, which we also denote by $\mu$.

We consider the metric $d_N$ on the tree $\Sigma_N$ given by
\[
d_N(\underline a,\underline b)=2^{-\max\{j:a_1\cdots a_{n_j}=b_1\cdots b_{n_j}\}}.
\] 
This metric depends on $\theta$ through the definition of the sequences $a_{n_j}, b_{n_j}$. For all sequences $\underline a,\underline b\in \Sigma$, and the corresponding sequences in $\Sigma_N$, 
$$
|\pi_{\theta}(\underline a)-\pi_{\theta}(\underline b)|\leq |\pi_{\theta}(D_{\underline a\wedge\underline b})|
\leq d_N(\underline a,\underline b)
$$
for all $N$. In particular the map $\pi_{\theta}:(\Sigma_N,d_N)\mapsto[-1,1]$ is Lipschitz. 

Let $\mu_{[a_1\cdots a_n]}:=\mu|_{[a_1\cdots a_n]}/\mu[a_1\cdots a_n]$. We state a variant of results of Hochman and Shmerkin \cite{HochShmerProj} that we will apply to the projections.

\begin{theorem}\label{HSThm}
If
\begin{equation}\label{eq0}
\lim_{N\to\infty} \liminf_{n\to\infty}\frac{1}{N\log 2}\frac{1}{n}\sum_{i=1}^n H_{2^{-(i+1)N}}\big(\pi_{\theta}(\mu_{[a_1\cdots a_{n_i}]})\big)\geq C
\end{equation}
for $\mu$-almost every $\underline a\in\Sigma_N$ then $\dim_H \pi_{\theta}(\mu)\geq C$.
\end{theorem}

This statement is essentially contained in Theorems 4.4 and 5.4 of \cite{HochShmerProj}.  However a little work is required to translate those theorems to our setting and the proof is given in the appendix.

The rest of the proof of Theorem \ref{MainThm} depends on studying the limits in (\ref{eq0}). Let 
\[\lambda_{\max}:=\max\{|a_i|_{\theta}:i\in\{1,\cdots,k\}, \theta\in\mathbb P\mathbb R^1\}.\]
 
\begin{lemma}\label{entest1}
For all $\underline a\in\Sigma$, all  $i, N\in\mathbb N$ and all $\theta \in \mathbb P\mathbb R^1$,
\begin{equation}\label{entest}
H_{2^{-(i+1)N}}(\pi_{\theta}(\mu_{[a_1\cdots a_{n_i}]}))\geq H_{2^{\lambda_{\max}}2^{-N}}(\pi_{\phi_{a_{n_i}\cdots a_1}(\theta)}(\mu)).
\end{equation}
\end{lemma}
\begin{proof}

Note that
\[
\pi_{\theta}(\mu_{[a_1\cdots a_{n_i}]})= S(\pi_{\phi_{a_{n_i}\cdots a_1}(\theta)}(\mu)),
\]
where $S$ is the map on measures on $[-1,1]$ induced by a linear map on $[-1,1]$ which contracts by a factor  $2^{-|a_1\cdots a_n|_{\theta}}$. An entropy $H_{\lambda}(\tau)$ is unchanged by rescaling the measure $\tau$ provided that we rescale the parameter $\lambda$ by the same amount. From the definitions of $\lambda_{\max}$ and $n_k$,
\[
|a_1\cdots a_{n_i}|_{\theta}\in[iN,iN+\lambda_{\max}].
\] 

Then
\begin{eqnarray*}
 H_{2^{-(i+1)N}}(\pi_{\theta}(\mu_{[a_1\cdots a_{n_i}]}))&=& H_{\frac{2^{-(i+1)N}}{2^{-|a_1\cdots a_{n_i}|_{\theta}}}}(\pi_{\phi_{a_{n_i}\cdots a_1}(\theta)}(\mu))\\
&=& H_{\rho.2^{-N}}(\pi_{\phi_{a_{n_i}\cdots a_1}(\theta)}(\mu)).
\end{eqnarray*}
where the `error' $\rho$ is given by
\[
\rho=\frac{2^{iN}}{2^{-|a_1\cdots a_{n_i}|_{\theta}}}\in[1, 2^{\lambda_{\max}}].
\]
As \eqref{nk}  $\rho\leq 2^{\lambda_{\max}}$,  where $\lambda_{max}$ is independent of $N$, and the entropy $H_{\lambda}(\mu)$ is monotone decreasing in $\lambda$, we obtain \eqref{entest}.
\end{proof}

We need to pass from equidistribution results with respect to the sequence $(\phi_{a_n\cdots a_1}(\theta))_{n=1}^{\infty}$ to corresponding results for the subsequence $(\phi_{a_{n_i}\cdots a_1}(\theta))_{i=1}^{\infty}$.

 
\begin{prop}\label{SuspensionProp}
There exists a measure $\nu_F$ on $\mathbb{PR}^1$ which is equivalent to $\mu_F$ and such that for all $\theta\in B$, $N\in\mathbb N$ and $\mu$-almost all $\underline a\in\Sigma$, the sequence $(\phi_{a_{n_i}\cdots a_1}(\theta))_{i=1}^{\infty}$ equidistributes with respect to $\nu_F$, i.e. for all intervals $A\subset \mathbb{PR}^1$ we have
\[
\lim_{I\to\infty}\frac{1}{I}\Big|\big\{i\in\{1,\cdots, I\}:\phi_{a_{n_i}\cdots a_1}(\theta)\in A\big\}\Big|=\nu_F(A).
\]
\end{prop}
The proof of this proposition is relatively long but involves only standard ergodic theory so is postponed to the appendix.

We now study the sums in \eqref{eq0}.
\begin{theorem}\label{Thm2}
For all $\theta\in B$ and $\mu$-almost every $\underline a\in\Sigma$,
\begin{equation}\label{thm2ineq}
\lim_{N\to\infty}\liminf_{n\to\infty}\frac{1}{N\log 2}\frac{1}{n}\sum_{i=1}^n H_{2^{-(i+1)N}}(\pi_{\theta}(\mu_{[a_1\cdots a_{n_i}]}))\geq \beta(\mu).
\end{equation}

\end{theorem}
\begin{proof}
By Lemma \ref{entest1}, for  $\theta\in B$ and 
 $\mu$ almost every $\underline a\in\Sigma$,
\begin{align}\label{Ineq1}
\liminf_{n\to\infty}\frac{1}{N\log 2}\frac{1}{n}\sum_{i=1}^n & H_{2^{-(i+1)N}}(\pi_{\theta}(\mu_{[a_1\cdots a_{n_i}]}))\nonumber\\
&\geq  \liminf_{n\to\infty} \frac{1}{N\log 2}\frac{1}{n}\sum_{i=1}^n H_{2^{-(N-\lambda_{\max})}}(\pi_{\phi_{a_{n_i}\cdots a_1}(\theta)}(\mu))\nonumber\\ 
&\geq  \frac{1}{N\log 2}\int_{\mathbb P\mathbb R^1} H_{2^{-(N-\lambda_{\max})}}(\pi_{\alpha}(\mu)) d\nu_F(\alpha).
\end{align}
where the second inequality holds because the function
\[
\alpha\mapsto H_{2^{-N}}(\pi_{\alpha}(\mu))
\]
is lower semi-continuous in $\alpha$ and the sequence $(\phi_{a_{n_i}\cdots a_1}(\theta))_{i=1}^{\infty}$ equidistributes with respect to $\nu_F$ by Proposition \ref{SuspensionProp}. 

For $\mu_F$ almost every $\alpha$ the measure $\pi_{\alpha}(\mu)$ is exact dimensional with dimension $\beta(\mu)$, and since $\mu_F$ and $\nu_F$ are equivalent the same is true for $\nu_F$ almost all $\alpha$. By \eqref{exactent}
\begin{equation}\label{emtbet}
\lim_{N\to\infty}\frac{1}{(N-\lambda_{\max})\log 2} H_{2^{-(N-\lambda_{\max})}}(\pi_{\alpha}(\mu))=\beta(\mu)
\end{equation}
for $\nu_F$ almost every $\alpha$.

Noting that $\lim_{N\to\infty}\frac{N}{N-\lambda_{max}}=1$, we may integrate \eqref{emtbet} over $\theta$ to get
\[
\int_{\mathbb P\mathbb R^1}\lim_{N\to\infty} \frac{1}{N\log 2} H_{2^{-(N-\lambda_{\max})}}(\pi_{\alpha}(\mu)) d\nu_F(\alpha)=\beta(\mu).
\]
Since this integrand is bounded, we may interchange the limit and integral  by the dominated convergence theorem. Combining this with inequality (\ref{Ineq1}) gives \eqref{thm2ineq}.\end{proof}

Putting together Theorems \ref{HSThm} and \ref{Thm2} we conclude that, for all $\theta\in B$,
\[
\dim_H\pi_{\theta}(\mu)\geq \lim_{N\to\infty}\lim_{n\to\infty}\frac{1}{N\log 2}\frac{1}{n}\sum_{i=1}^n H_{2^{-(i+1)N}}(\pi_{\theta}(\mu_{[a_1\cdots a_{n_i}]}))\geq \beta(\mu)
\]
which completes the proof of Theorem \ref{MainThm}.

\section{Appendix: Technical Proofs}
\setcounter{equation}{0}
\subsection{Proof of Theorem \ref{HSThm}}
We use the terms `tree morphism' and `faithful map' as defined in \cite{HochShmerProj}: essentially a tree morphism is a map from one sequence space to another which preserves the metric and the structure of cylinder sets, and a faithful map is a map (in our case from a sequence space to $[-1,1]$) which does not distort dimension too much. The interested reader can find technical definitions of these terms, along with many lemmas of Hochman and Shmerkin which we use but do not write out in full, in \cite{HochShmerProj}.

\begin{proof}
First we define a space $Y_N$ and split the projection $\pi_{\theta}:\Sigma_N\to[-1,1]$ into a tree-morphism $g_N:\Sigma_N\to Y_N$ which is easier to deal with than $\pi_{\theta}$ itself, and a faithful map $h_N:Y_N\to[-1,1]$.

Let $Y_N:=\{-2^{N+1},\cdots 2^{N+1}-1\}^{\mathbb N}$ be equipped with metric $d_N$. Associate to each $i\in\{-2^{N+1},\cdots,2^{N+1}-1\}$ an interval $I_i$ of length $2^{-N}$ with left endpoint at $\frac{i}{2^{N+1}}$. This gives an overlapping covering of $[-1,1]$, with each point $x\in[0,1]$ contained in either two or three intervals $I_i$. Let $S_i:[-1,1]\to[-1,1]$ be the linear contraction mapping $[-1,1]$ onto $I_i$, and let $h_N:Y_N\to [-1,1]$ be given by
\[
h_N(\underline a)=\lim_{n\to\infty} S_{a_1}\circ\cdots\circ S_{a_n}(0)\qquad (\underline a \in \Sigma).
\] 
The map $h_N$ is 3-faithful in the sense of \cite{HochShmerProj} since for each point $x$ in each interval $h_N[a_1\cdots a_n]$ there are at most three values of $a_{n+1}\in \{0,\cdots, 2^{N+1}-1\}$ for which $x\in h_N[a_1\cdots a_{n+1}]$. 
By \cite[Proposition 5.2]{HochShmerProj}  there exists a constant $C_2$, independent of $N$, such that for every measure $\nu$ on $Y_N$
\begin{equation}\label{HS5.2}
\dim_H \nu-\frac{C_2}{N\log 2}\leq\dim_H h_N(\nu).
\end{equation}

We now define a map $g_N:\Sigma_N\to Y_N$ which maps depth-$n$ cylinders in $\Sigma_N$ to depth-$n$ cylinders in $Y_N$.
Note that, by the arrangement of the intervals $I_i$, for each interval $A\subset[-1,1]$ of radius less than $2^{-(j+1)N}$ there is a choice of $b_1\cdots b_j$ such that $A\subset h_N[b_1\cdots b_j]$. Furthermore, for every $A'\subset A$ of radius less than $2^{-(n+1)N}$ where $n>j$, there is an extension $b_{j+1}\cdots b_n$ such that $A'\subset h_N[b_1\cdots  b_j b_{j+1}\cdots b_n]$.

We define $g_N$ iteratively. First we choose for each depth-$2$ cylinder $[a_1\cdots a_{n_2}]\subset \Sigma_N$ a depth-1 cylinder $[b_1]\subset Y_N$ such that $\pi_{\theta}(D_{a_1\cdots a_{n_2}})\subset h_N([b_1])$. 
Then for each depth-$(j+1)$ cylinder $[a_1\cdots a_{n_{j+1}}]\subset \Sigma_N$, having already chosen a depth-$(j-1)$ cylinder $[b_1\cdots b_{j-1}]$ corresponding to word $a_1\cdots a_{n_j}$, we choose a letter $b_j$ such that 
\[
\pi_{\theta}(D_{a_1\cdots a_{n_{j+1}}})\subset h_N([b_1\cdots b_j]).
\]
This process defines a tree morphism $g_N: \Sigma_N\to Y_N$ such that $\pi_{\theta}:\Sigma_N\to[-1,1]$ can be written $\pi_{\theta}=h_N\circ g_N$.

Since $h_N$ is faithful and $\pi_{\theta}=h_N\circ g_N$, 
\[
\big|H_{2^{-(i+1)N}}\big(\pi_{\theta}(\mu_{[a_1\cdots a_{n_i}]})\big)-H_{2^{-(i+1)N}}\big(g_N(\mu_{[a_1\cdots a_{n_i}]})\big)\big|
\]
is bounded above by some constant which is independent of $N$. In particular, inequality \eqref{eq0} holds if and only if
\[
\lim_{N\to\infty} \liminf_{n\to\infty}\frac{1}{N\log 2}\frac{1}{n}\sum_{i=1}^n H_{2^{-(i+1)N}}(g_N(\mu_{[a_1\cdots a_{n_i}]}))\geq C.
\]
But  since $g_N$ is a tree morphism, \cite[Theorem 4.4]{HochShmerProj} gives
\[
\lim_{N\to\infty} \dim_H g_N(\mu)\geq C,
\]
so setting $\nu= g_N(\mu)$ in (\ref{HS5.2}),
\[
\dim_H \pi_{\theta}(\mu)\geq C.
\]
\end{proof}

\subsection{Proof of Proposition \ref{SuspensionProp}}
We begin by defining a map $T:\Sigma\times\mathbb {PR}^1\to\Sigma\times\mathbb{PR}^1$ by
\[
T(\underline a,\theta)=(\sigma(\underline a), \phi_{a_1}(\theta)).
\]
We can map the two-sided full shift $\Sigma^{\pm}$ onto $\Sigma\times \mathbb {PR}^1$ by mapping two-sided sequence $\underline a^{\pm}$ to $(a_1a_2a_3\cdots, \lim_{n\to\infty} \phi_{a_0\cdots a_{-n}}(0))$. Then $T$ is a factor of the right shift $\sigma^{-1}$ under this factor map, and the two-sided extension $\overline{\mu}$ of $\mu$ is mapped to $\mu\times\mu_F$. Since factors of mixing dynamical systems are mixing,  the system $(\Sigma\times\mathbb{PR}^1,\mu\times\mu_F,T)$ is mixing.
The map $T$ helps us understand the sequence $(\phi_{a_{n_i}\cdots a_1}(\theta))_{i=1}^{\infty}$, but we need to know about the sequence $n_i=n_i(\underline a,\theta,N)$. For this we build a suspension flow over $T$ by introducing a notion of time. 

Consider the space 
\[
W:=\{(\underline a,\theta,t)\in \Sigma\times\mathbb{PR}^1\times\mathbb R :0\leq t\leq |a_1|_{\theta}\}
\]
where the points $(\underline a,\theta,|a_1|_{\theta})$ and $(\sigma(\underline a),\phi_{a_1}(\theta),0)$ are identified. Define a suspension semiflow $\psi$ on $W$ by letting
\[
\psi_s(\underline a,\theta,t)=(\underline a,\theta,t+s)
\]
for $0\leq t+s\leq |a_1|_{\theta}$, and extending this to a well-defined semiflow for all $s>0$, using the identification $(\underline a,\theta,|a_1|_{\theta})=(\sigma(\underline a),\phi_{a_1}(\theta),0)$. This flow preserves the measure $(\mu\times\mu_F\times\mathcal L)|_{W}$, where $\mathcal L$ is Lebesgue measure. Since suspension semiflows over ergodic maps are ergodic,  $\psi_s$ is ergodic.

We now prove some uniform continuity lemmas involving $\theta, \theta' \in  \mathcal Q_2$, which also give non-uniform continuity for $\theta, \theta' \in  B$, since all $\theta\in B$ are eventually mapped into $\mathcal Q_2$ by compositions of the $\phi_i$.

We first show that small variations of $\theta$ have little effect on the time taken to flow through the base a given number of times.

\begin{lemma}
There exists a constant $C$ such that for all $\theta,\theta'\in\mathcal Q_2, n\in\mathbb N$ and $\underline a\in\Sigma$,
\[
|a_1\cdots a_n|_{\theta}-|a_1\cdots a_{n}|_{\theta'}<C|\theta-\theta'|
\]
\end{lemma}
\begin{proof}
For each $i\in\{1,\cdots,k\}$, the map $\theta\mapsto |a_i|_{\theta}$ is differentiable with derivative bounded above by 2, since sin and cos have derivatives bounded by $1$. Furthermore, there exists a constant $\rho<1$ such that the maps $\phi_i$ restricted to $\mathcal Q_2$ are strict contractions with derivative bounded above by $\rho$.\begin{eqnarray*}
|a_1\cdots a_n|_{\theta}-|a_1\cdots a_{n}|_{\theta'}&\leq & \sum_{i=1}^{n} \big||a_i|_{\phi_{a_{n-1}\cdots a_1}(\theta)}-|a_i|_{\phi_{a_{n-1}\cdots a_1}(\theta')}\big|   \\
&\leq& \big(2\max_{i,\theta''} |a_i|_{\theta''}\big)\sum_{i=1}^{n} \big|\phi_{a_{n-1}\cdots a_1}(\theta)-\phi_{a_{n-1}\cdots a_1}(\theta')\big|\\
&\leq & \big(2\max_{i,\theta''} |a_i|_{\theta''}\big) |\theta-\theta'|\Big(\sum_{i=1}^{n} \rho^{i-1}\Big)\\ 
&<& C|\theta-\theta'|
\end{eqnarray*}
for all $\theta,\theta'\in\mathcal Q_2$.
\end{proof}

\begin{lemma}\label{time}
Suppose that the orbit under $\psi$ of point $(\underline a,\theta,t)$ equidistributes with respect to the measure $(\mu\times\mu_F\times\mathcal L)|_{W}$. Then the same is true of $(\underline a,\theta',t')$ for all $\theta'\in \mathcal Q_2$ and $t'\in[0, |a_1|_{\theta'}]$.
\end{lemma}

\begin{proof}
The distance $|T^n(\underline a,\theta)-T^n(\underline a,\theta')|\to0$ as $n\to\infty$ since the maps $\phi_i$ are strict contractions on $\mathcal Q_2$. All we need to check is that there is not too much time distortion when we replace the transformation $T$ with the flow $\psi$. But this has been dealt with by the previous lemma. 
\end{proof}

Given a flow $\psi_s$ on $W$, the {\it time-$N$ map} is the map $\psi_N$ on $W$ regarded as a discrete time dynamical system.

\begin{lemma}
The time N-map $\psi_N:W\to W$ preserves the measure $(\mu\times\mu_F\times\mathcal L)|_{W}$ and $(W, \psi_N, (\mu\times\mu_F\times\mathcal L)|_{W})$ is ergodic.
\end{lemma}

\begin{proof}
First we deal with the case when there exists an irrational number $c>0$ such that every periodic orbit of the suspension flow has period equal to an integer multiple of $c$. Then it was argued in \cite[Proposition 5]{QuasSoo}, using the Livschitz Theorem, that the roof function of our flow is cohomologous to a function which always takes values equal to an integer multiple of $c$. We add this coboundary, which has no effect on the dynamics of the flow or the time $N$ map, and assume now that our roof function always takes value equal to an integer multiple of $c$. Now using the fact that the base map $T$ of our suspension flow is mixing, coupled with the fact that the map $x\mapsto x+N$ $($mod $c)$ is ergodic, it follows that the map $\psi_N$ is ergodic.

In the case that there is a rational number $c>0$ such that every periodic orbit has period equal to an integer multiple of $c$, we replace $\log_2$ with $\log_3$ in the definition of our distance function $|.|_{\theta}$ which makes the roof function. Then the corresponding constant for our new suspension flow is irrational, and apply the previous argument in this setting.

Finally we deal with the case that there exists no $c>0$ such that every periodic orbit of the suspension flow has period equal to an integer multiple of $c$. In \cite[Proposition 5]{QuasSoo} it is shown that this ensures that the flow is weak mixing. But if the suspension flow $\psi$ is weak mixing then the corresponding time-N map $\psi_N$ is ergodic as required. \end{proof}

Finally we complete the proof of  Proposition \ref{SuspensionProp}. Let $\pi_2:W\to \mathbb{PR}^1$ be given by $\pi_2(\underline a,\theta,t)=\theta$, and let the measure $\nu_F$ on $\mathbb{PR}^1$ be given by $\nu_F=\big((\mu\times\mu_F\times\mathcal L)|_{W}\big)\circ \pi_2^{-1}$. This measure is equivalent to $\mu_F$.

If $\psi_N(\underline a,\theta,0)$ equidistributes with respect to $(\mu\times\mu_F\times\mathcal L)$ then the sequence $\big(\phi_{a_{n_i}\cdots a_1}(\theta)\big)_{i=1}^{\infty}$ equidistributes with respect to $\nu_F$. This is because, from the definition of the sequence $a_{n_i}$, 
\[
\phi_{a_{n_i-1}\cdots a_1}(\theta)=\pi_2\big(\psi_{iN}(\underline a,\theta,0)\big).
\]
Then it is enough to prove that for all $\theta\in B$, for $\mu$-almost all $\underline a$, the sequence $\big(\psi_{iN}(\underline a,\theta,0)\big)_{i=0}^{\infty}$ equidistributes with respect to $\nu_F$. We have already argued that this holds for $(\mu\times\mu_F)-$almost all pairs $(\underline a,\theta)$, so the extension to all $\theta$ follows readily from Lemma \ref{time} and the fact that for all $\theta\in B$, for $\mu$-almost every $\underline a$, the orbit $\big(\phi_{a_n\cdots a_1}(\theta)\big)$ gets arbitrarily close to some $\mu_F$-typical point of $\mathbb{PR}^1$. 



\bibliographystyle{plain} 
\bibliography{SelfAffine.bib}

\end{document}